\documentclass[AMA,STIX1COL]{WileyNJD-v2}

\articletype{Research Article}%

\received{2022}
\revised{ }
\accepted{2023}
\doiheadtext{10.1002/mma.9220}

\usepackage{bm}
\usepackage{subfigure}

\DeclareMathOperator{\supp}{supp}
\DeclareMathOperator{\sgn}{sgn}

\usepackage{hyperref}
\hypersetup{
    colorlinks=true,
    linkcolor=blue,
    filecolor=blue,      
    urlcolor=blue,
    breaklinks=true
}

\usepackage[capitalise]{cleveref}
\crefname{axiom}{Axiom}{Axioms}
\crefname{lemma}{Lemma}{Lemmas}
\crefname{remark}{Remark}{Remarks}
\crefname{corollary}{Corollary}{Corollaries}
\crefname{assumption}{Assumption}{Assumptions}
\crefname{theorem}{Theorem}{Theorems}

\begin{document}

\title{Nonlinear Einstein paradigm of Brownian motion and localization property of solutions}

\author[1]{Ivan C. Christov}

\author[2]{Isanka Garli Hevage}

\author[2,3]{Akif Ibraguimov}

\author[2]{Rahnuma Islam}

\authormark{CHRISTOV \textsc{et al}.}

\address[1]{\orgdiv{School of Mechanical Engineering}, \orgname{Purdue University}, \orgaddress{West Lafayette, \state{Indiana} 47907, \country{USA}}}

\address[2]{\orgdiv{Department of Mathematics \& Statistics}, \orgname{Texas Tech University}, \orgaddress{Lubbock, \state{Texas} 79409, \country{USA}}}

\address[3]{\orgdiv{Oil and Gas Research Institute}, \orgname{Russian Academy of Sciences}, \orgaddress{Moscow, \country{Russia}}}

\corres{\email{christov@purdue.edu}, 
 \email{isankaupul.garlihevage@ttu.edu}, 
 \email{Akif.Ibraguimov@ttu.edu}, 
 \email{rahnuma.islam@ttu.edu}}

\abstract[Summary]{We employ a generalization of Einstein's random walk paradigm for diffusion to derive a class of multidimensional degenerate nonlinear parabolic equations in non-divergence form. Specifically, in these equations, the diffusion coefficient can depend on both the dependent variable and its gradient, and it vanishes when either one of the latter does. It is known that solutions of such degenerate equations can exhibit finite speed of propagation (so-called localization property of solutions). We give a proof of this property using a De Giorgi--Ladyzhenskaya iteration procedure for non-divergence-from equations. A mapping theorem is then established to a divergence-form version of the governing equation for the case of one spatial dimension. Numerical results via a finite-difference scheme are used to illustrate the main mathematical results for this special case. For completeness, we also provide an explicit construction of the one-dimensional self-similar solution with finite speed of propagation function, in the sense of Kompaneets--Zel'dovich--Barenblatt. We thus show how the finite speed of propagation quantitatively depends on the model's parameters.}



\jnlcitation{\cname{%
\author{I.\ C.\ Christov}, 
\author{I.\ Garli Hevage}, 
\author{A.\ Ibraguimov}, and
\author{R.\ Islam}} (\cyear{2023}), 
\ctitle{Nonlinear Einstein paradigm of Brownian motion and localization property of solutions}, \cjournal{Math.\ Meth.\ Appl.\ Sci.}, doi:10.1002/mma.9220.}

\maketitle


\section{Introduction}
\label{sec:intro}

A number of physical processes are governed by nonlinear parabolic partial differential equations (PDEs), specifically ones in which the coefficient of the highest-order spatial derivative (identified as the \emph{diffusion coefficient} in the case of a parabolic PDE) can vanish  \cite{DiBenedetto93,Vazquez07}. These \emph{degenerate} parabolic equations can exhibit finite speed of propagation, \textit{i.e.}, initial conditions of compact support expand the length of their support in time at a finite speed (see, \textit{e.g.}, \S1.9 of Straughan's book \cite{S11}). This property is in contrast to the case of linear parabolic equations, for which initial data ``propagates'' instantaneously throughout the domain, leading to the so-called \emph{paradox of heat conduction} \cite{Keller2004}. More generally, compact (localized) solutions are but one example of traveling waves in nonlinear diffusion-convection-reaction equations \cite{GK04}. 

In physics and in engineering, nonlinear degenerate diffusion equations have been derived to describe the flow and spreading of compressible fluids (gases) through porous media \cite{Barenblatt52,Philip70}. The evolution of the thickness of thin liquid films on solid surfaces \cite{Oron97}, determined by the balance of viscous flow forces, surface tension, and van der Waals interactions, is also governed by a degenerate nonlinear parabolic equation, now of higher order \cite{Bertozzi96,Witelski20}. High-temperature phenomena, such as shock waves in ionized gases, can also be described by nonlinear degenerate diffusion equations, in which the thermal diffusivity vanishes at a certain temperature  \cite{Zel67}. In these previous examples, the diffusion coefficient may only depend on (and vanish with) the dependent variable (say, gas concentration, fluid film height, or temperature). Interestingly, there are also physics and engineering problems in which the diffusivity can also vanish when the \emph{gradient} of the dependent variable vanishes. One example is pre-Darcy flow in saturated porous media \cite{Celik17}. Another example is the spreading of confined layers of non-Newtonian fluids \cite{Ciriello2016,DiFed2017}, for which the complex rheology of the fluid leads to a degenerating gradient-dependent diffusivity (unlike the Newtonian case \cite{Oron97}); see also the review by Ghodgaonkar and Christov \cite{gc19}. Indeed, degenerate equations with gradient-dependent diffusivity are traditionally associated with problems of non-Newtonian filtration through a porous medium (see, \textit{e.g.}, the discussion by Kalashnikov\cite{Kalashnikov1987}).

In the present work, we reconsider all of the above models, and their governing nonlinear parabolic PDEs, from the point of view of Einstein's random walk paradigm \cite{Einstein05,Einstein56} for diffusion in a continuous medium. Specifically, we show how to derive a generic degenerate nonlinear parabolic equation, in non-divergence form, with the diffusivity depending on \emph{both} the dependent variable and its gradient. This approach can be interpreted as a complementary conceptual derivation of the governing equation for the physical situations mentioned above, but \emph{without} appealing to Fick's (or Darcy's or Fourier's) law (see, \textit{e.g.}, Ch.~1 of the book by Cunningham \cite{Cunningham1980}), the continuity (conservation of mass) equation, and the thermodynamic closures that relate the various system variables (such as density, mass, and concentration) to each other.

To this end, this paper is organized as follows. In \cref{sec:paradigm}, we introduce the Einstein random-walk paradigm and its generalization used for the present derivations. In \cref{sec:max_bd}, we prove a maximum principle for the solution of the multidimensional nonlinear parabolic equation derived, taking into account its degeneracy with respect to the dependent variable and its gradient. To accomplish this task, we regularize the degeneracy of the PDE, proving the maximum principle for a smooth viscosity solution, then we take the limit. Next, in \cref{sec:localization}, we strengthen our result by proving the localization property of solutions using a De Giorgi--Ladyzhenskaya iteration procedure, \emph{without} the need for explicitly constructing a barrier function. Again, we work with the regularized problem and in terms of the  viscosity solution, for convenience, but the final estimate is independent of the (small) regularization parameter. In \cref{sec:mapping}, for the case of one spatial dimension, we develop transition formulas to map the non-divergence form of the PDE (discussed in earlier sections) to the more traditional divergence-form case. Beyond showing a fundamental equivalence of the two forms of the PDE, the latter is more convenient for numerical simulation, since we have previously developed and validated a finite-difference scheme for divergence-form nonlinear degenerate PDEs \cite{gc19}. Then, in \cref{sec:numerical} again for the special case of one spatial dimension, we construct explicit self-similar solutions with finite speed of propagation and perform simulations of the divergence-form degenerate PDE to illustrate and verify our main mathematical result from \cref{sec:localization}. Finally, \cref{sec:discussion} discusses the broader context of our results and avenues for future work.


\section{Einstein paradigm for diffusion}
\label{sec:paradigm}

\subsection{Setup and preliminaries}
\label{sec:setup}

In the celebrated work of Einstein \cite{Einstein05}, Brownian motion was explained using a mathematical model based on a thought experiment. Since then, the stochastic interpretation of this approach has been applied to many transport processes arising in physics, chemistry, and engineering (see, \textit{e.g.}, \S1.2.1 of Gardiner's book \cite{G09}). 

To formulate the PDE model (following our earlier work \cite{CII20}), consider a collection of  ``particles,'' which here stands for any discrete physical object and some physics governing its motion with respect to nearby particles. The particles may be actual particles, or molecules, or physical agents or even animals \cite{MCB}, depending on the physical process under consideration. The modeling approach under the Einstein paradigm is generic.

To begin, first we require that there exists a time interval $\tau$ such that (s.t.), during two consecutive intervals $\tau$, the motions performed by individual particles in the system are mutually independent (uncorrelated) events. The interval $\tau$ is ``sufficiently small'' compared to the time scale $t$ of observation of the physical process, but not so small that the motions become correlated. 

Now, we will extend the traditional one-dimensional Einstein model of Brownian motion to the $d$-dimensional real Euclidean space $\mathbb{R}^{d}$. Let $x \in \mathbb{R}^{d} $ and  $u(x,t):\mathbb{R}^d\times[0,\infty)\to\mathbb{R}$. We assume that the number of particles per unit volume at point $x$ at time $t$ can be represented by (\textit{i.e.}, it is proportional to) the concentration, which is the scalar function $u(x,t)$.
\begin{assumption}\label{ext-Eins}
We assume the following extension of the axioms formulated by Einstein \cite{Einstein05,Einstein56}:
\begin{enumerate}
\item Consider a particle ($P$) of particular type moving through the medium of interest. Let $\mathbb{P}(\tau)$ be the set of vectors with non-colliding jumps of $P$ during the time interval $\tau$. We call  $\vec{\Delta} 
= \big(\Delta_1, \hdots ,\Delta_d\big)^\top$ 
the ``vector of free jumps of particle $P$'' if $\vec{\Delta} \in  \mathbb{P}(\tau)$.
\item The interactions of particles during the time interval $\tau$ can occur via absorption (or reaction) through the surrounding medium, via collisions with other particles, and possibly via the spatial domain's boundaries. In general, this key quantity $\tau$ can depend on the concentration of particles $u$, its gradient $\nabla u$, and possibly also $x$ and $t$.  
\item The time interval of free jumps $\tau$, the vector of expected free jump lengths $\vec{\Delta}^e$, and the scalar probability density function $\varphi(\vec{\Delta})$ of free jumps $\vec{\Delta}$ are the quantities that characterize the generalized Brownian process under the Einstein paradigm. Here, the free jumps in each coordinate direction are denoted $\Delta_i$ for $i=1,2,\hdots, d$. 
\item During the time interval $[t,t+\tau]$ in the unit volume around the particle $P$ located at the observation point $x$, it is  possible absorption and/or reaction with other particles (or with the suspending medium) occurs. This effect is represented by the integral
\begin{equation}
    \int_t^{t+\tau}\mathcal{A}\big(u(x, s),s\big) \, ds,
\end{equation}
where $\mathcal{A}(u,t)$ is a potentially nonlinear function that can be positive or negative, depending on the physics considered.
\end{enumerate}
\end{assumption}

\begin{axiom}[Whole universe axiom]
 \begin{equation}\label{uni-ax}
      \int_{\mathbb{P}(\tau)}\varphi(\vec{\Delta}) \,d\vec{\Delta} = 1,
 \end{equation}
 where $d\vec{\Delta} \triangleq  d\Delta_1 d\Delta_2 \cdots d\Delta_d$.
\end{axiom}
Note that, in a view of the definition of the set $\mathbb{P}(\tau)$, if $\vec{\Delta} \notin \mathbb{P}(\tau)$, then $\varphi(\vec{\Delta})=0$.

Let us define the vector of expected free jumps:
\begin{equation}\label{exp-L}
\vec{\Delta}^e \triangleq (\Delta^{e}_{1}, \Delta^{e}_{2}, \hdots, \Delta^{e}_{d})^\top, \quad   \text{where} \quad  \Delta^{e}_{i} \triangleq \int_{\mathbb{P}(\tau)}{\Delta_{i}} \varphi(\vec{\Delta}) \,d\vec{\Delta},
\end{equation}
and the matrix of standard (co-)variances of free jumps:
\begin{equation}\label{var-ij}
\sigma_{ij}^2 \triangleq  \int_{\mathbb{P}(\tau)} \left( {\Delta_{i}}-{\Delta}^{e}_{i} \right) \left( {\Delta_{j}}-{\Delta}^{e}_{j} \right) \varphi(\vec{\Delta}) \,d\vec{\Delta}.
\end{equation}
Evidently $\vec{\Delta}_{e}(x,t)$ and $\sigma_{ij}(x,t)$ depend on space $x$ and time $t$. Next, considering \cref{ext-Eins} and the above, we state the generalized Einstein paradigm axiom for the number of particles found at time $t+\tau$ in an infinitesimal control volume $dv \subset\mathbb{R}^d$ containing the point $x$:
\begin{axiom}[Einstein conservation law]\label{Einstein_conserv_axiom}
\begin{equation}\label{Einstein_conserv_eq}
u(x, t+\tau) \cdot dv =  
\Bigg(\,\underbrace{\int_{\mathbb{P}(\tau)} u\big(x+ \vec{\Delta}, t\big) \varphi(\vec{\Delta}) \,d\vec{\Delta}}_{\text{free jumps}}
+ 
\underbrace{\int_t^{t+\tau}\mathcal{A}\big(u(x, s),s\big) \,ds}_{\text{absorption/reaction}}
\Bigg)\cdot dv .
\end{equation}
\end{axiom}
The Einstein conservation law (\cref{Einstein_conserv_axiom}) is  illustrated schematically in \cref{fig:schematic}.

\begin{figure}
\centering
\includegraphics[scale=0.75]{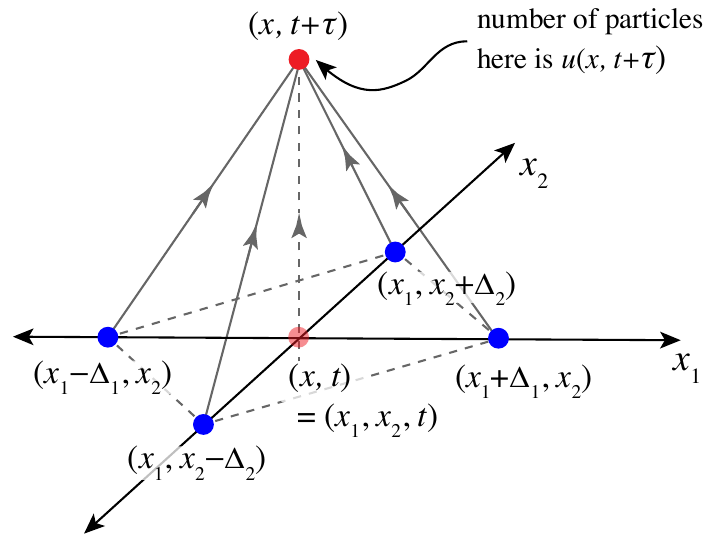}
\caption{Schematic illustrating the notation and jump process, from $t$ to $t+\tau$, for the proposed generalized Einstein random walk paradigm for $x=(x_1,x_2)\in\mathbb{R}^2$.}
\label{fig:schematic}
\end{figure}

Assuming that $u(x,t)\in \mathcal{C}^{2,1}_{x,t}$ (\textit{i.e.}, $u$ is twice continuously differentiable in $x$ and once in $t$) and using the multidimensional analog of the Caratheodory Theorem \cite{Bartle} and arguments from our previous work \cite{CII20}, Einstein's conservation equation \eqref{Einstein_conserv_eq} can be reduced to a generic advection-reaction-diffusion equation:
\begin{equation}\label{dif-drift-absorp-generic}
   Lu=\tau\frac{\partial u}{\partial t} + \underbrace{\vec{\Delta}^e\cdot\nabla u}_{\text{drift/advection}} - \underbrace{\sum_{i,j=1}^d a_{i,j}\frac{\partial^2 u}{\partial x_i \partial x_j}}_{\text{diffusion}} + \underbrace{\tau \mathcal{A}\big(u(x,t),t\big)}_{\text{absorption/reaction}} = 0.
\end{equation}

\begin{remark}
\Cref{dif-drift-absorp-generic} is applicable in the general case of an anisotropic medium with vector $\vec{\Delta}^e(x,t)$ of expected values of free jumps in the $d$ spatial directions, and symmetric matrix $\left(a_{i,j}\right)$ of coefficients 
\begin{equation}
a_{i,j} \triangleq 
\begin{cases}
\frac{1}{2}\sigma_{i,i}^2, &\quad \text{if} \quad i=j,\\
\sigma_{i,j}^2, &\quad \text{if} \quad  i\neq j,
\end{cases}
\end{equation}
which are defined in terms of the corresponding (co-)variances $\sigma_{i,j}$ of the free jump distribution $\varphi$, given in \cref{var-ij}.
\end{remark}

\subsection{Nonlinear PDE under the Einstein paradigm}

For the remainder of this paper, we assume there is no absorption or reaction (meaning no creation or annihilation of particles), so $\mathcal{A}(u)\equiv0$. Next, we specialize the model to a $d$-dimensional domain in $\mathbb{R}^d$ under the assumption that the medium in which the diffusion process occurs is homogeneous and isotropic. By the homogeneity assumption, the frequency distribution $\varphi$ is a scalar function of the vector $\vec{\Delta}^e$, and  the vector of expected values is simply  $\vec{\Delta}^e = (\Delta^e, \hdots, \Delta^e)$. In other words, the expected values of jumps in the random walk are the same in all coordinate directions, and equal to the same given constant value $\Delta^e$. It follows that $a_{i,j}=\tfrac{\sigma^2}{2}\delta_{i,j}$, where $\sigma^2$ is the now-constant variance of the jump distribution, and $\delta_{i,j}$ is the Kronecker-$\delta$ symbol.

Now, we further motivate how the model equation, derived in \cref{sec:setup} under the generalized Einstein paradigm, can be used to obtain a nonlinear parabolic diffusion equation, in non-divergence form, for a Brownian-like process.

\begin{assumption}[Time interval in nonlinear Einstein paradigm]\label{tau}
Under the Einstein para\-digm, $\tau$, $\Delta^e$, and $\varphi$ characterize the Brownian process. In general, they could be functions of both the spatial and time variables, $x$ and $t$, and properties of the medium. But, more importantly, they can be functions of the dependent variable, such as the concentration of particles $u\ge0$ and its derivatives, via a power-law relation (\cref{remark:u-dep,remark:grad-u-dep} motivate this functional form): 
\begin{equation}
    \tau=\frac{\tau_0}{u^{\gamma} |\nabla u|^{\beta}},
\end{equation}
where $\tau_0>0$ is a suitable dimensional consistency constant, and $\gamma,\beta\in\mathbb{R}$ do not have to be integers (though their ranges will be further restricted below for the proofs). Here, the earlier assumption of a homogeneous isotropic medium rules out an explicit dependence of $\tau$ on $x$ or $t$.
\end{assumption}

\begin{remark}[Classical Einstein paradigm]
Observe that $\tau_0$ (\textit{i.e.}, $\tau$ for $\gamma=\beta=0$) in \cref{tau} coincides with the time interval between jumps in the special case of Einstein's classical derivation \cite{Einstein05,Einstein56}.
\end{remark}

\begin{remark}[$u$ dependence of $\tau$]\label{remark:u-dep}
In ionized gases, the mean free path (between collisions) of molecules is directly proportional to a non-integer power of the temperature (see, \textit{e.g.}, Ch.~X\S2 of the book by Zel?dovich and Raizer \cite{Zel67}), with the power depending on the level of ionization. Therefore, kinematics implies that the time interval between collisions is also proportional to the temperature (because the ``agitation'' and kinetic energy of molecules increases with temperature). Translated to the context of concentration and particles: the length of free jumps, and therefore the time during which the free jumps occur, must be \emph{inversely} proportional to some power of the concentration. That is, the more particles there are nearby, the more likely they are to collide, thus reducing the time interval during which the free jumps occur (\textit{i.e.}, the time interval during which the motion is uncorrelated). This physical situation is a special case of \cref{tau} ($\beta=0$).
\end{remark}

\begin{remark}[$\nabla u$ dependence of $\tau$]\label{remark:grad-u-dep}
Similar arguments can also be made for the case of $\beta\ne0$ to justify that the time interval between collisions can also depend upon some (possibly non-integer) power of the gradient's magnitude $|\nabla u|$. For example, Kruckels \cite{Kruckels1973} argued that if a particle has more neighbors on one ``side'' (versus another), then a gradient effect arises because the preferentially large number of particles on one side will make it more likely for a collision to occur. In turn, this again reduces the time interval during which the free jumps occur, along the lines of \cref{remark:u-dep}. The power-law form of this gradient dependence has been further motivated by appealing to Eyring's theory of rate processes \cite{Kruckels1973,Cunningham1980}.
\end{remark}

Under the above assumptions, on a general domain $U \subset \mathbb{R}^d$, \cref{dif-drift-absorp-generic} becomes
\begin{equation}\label{non_diveregent_eq}
   Lu=\tau_0\frac{\partial u}{\partial t}+u^{\gamma}|\nabla u|^{\beta} \Delta^{e} \sum_{i=1}^{d} \frac{\partial u}{\partial x_i}-\frac{\sigma^2}{2} u^{\gamma}|\nabla u|^{\beta}\Delta u = 0 \quad \text{in} \;\; U\times (0,T], \;\; U \subset \mathbb{R}^d
\end{equation}
subject to an initial and boundary condition:
\begin{equation}\label{IBC}
    u(x,t)=h(x,t) \geq 0 \quad \text{ on } \ \Gamma(U_T). \ 
\end{equation}
Above, $\Gamma(U_T)$ is the parabolic boundary of $U_T=U\times (0,T]$.

\Cref{non_diveregent_eq} is sometimes called a \emph{doubly nonlinear} parabolic equation because it degenerates when $u=0$ \emph{and/or} $|\nabla u|=0$.


\section{Maximum principle on a bounded domain}
\label{sec:max_bd}

In this section, we prove a maximum principle for the solutions of \cref{non_diveregent_eq}, following \cite{Ilyin02,Landis,CII20}. In addition to $\mathbb{R}^d$ as above, we also consider the $(d+1)$-dimensional space $\mathbb{R}^{d+1}$, in which the spatial coordinates are augmented by time: $\mathbb{R}^{d+1}\ni (x,t) = (x_1, x_2,\hdots,x_d,t)$. 
Again, let $U \subset \mathbb{R}^{d}$ be bounded, and $t>0$. We define the cylindrical region $U_T \triangleq U\times (0,T] \subset \mathbb{R}^{d+1}$, its parabolic boundary $\Gamma(U_T) \triangleq (U \times \{t=0\})\cup (\partial U \times (0,T])$, and the layer $H\subset \mathbb{R}^{d+1}\cap(0,T]$. 

As above, $\mathcal{C}^{2,1}_{x,t}(U_T)$ denotes the class of continuous functions on $\overline{U_T}$ that have two continuous derivatives in $x$ and one continuous derivative in $t$ inside the domain $U_T$. 
\begin{remark}\label{epsil-sol}
Note that \cref{non_diveregent_eq} degenerates when $u=0$ or $|\nabla u|=0$, therefore its solution $u(x,t)$ does not belong to  $\mathcal{C}^{2,1}_{x,t}(U_T)$, strictly speaking. To prove our result, we introduce a ``viscous'' regularization. Specifically, we regularize the operator $L$ as $L_{\epsilon}$ via
\begin{equation}
L_{\epsilon}u_{\epsilon}=
   \tau_0\frac{\partial u_{\epsilon}}{\partial t}+(u_{\epsilon}+\epsilon)^{\gamma}(|\nabla u_{\epsilon}|^{\beta}+\epsilon) \Delta^{e} \sum_{i=1}^{d} \frac{\partial u_{\epsilon}}{\partial x_i}-\frac{\sigma^2}{2} (u_{\epsilon}+\epsilon)^{\gamma}(|\nabla u_{\epsilon}|^{\beta}+\epsilon)\Delta u_{\epsilon} = 0 
   \;\; \text{in} \;\; U\times (0,T], \;\; U \subset \mathbb{R}^d.
\end{equation}
\end{remark}

In all major estimates for the solution of the initial-boundary-value problem (IBVP) for $L_{\epsilon}u_{\epsilon}=0$ in this article, the constants do not depend on $\epsilon$. All dependence on the regularization parameter $\epsilon$ appears as separate terms in the respective estimates, and they do not depend on the solution $u_{\epsilon}$ or its derivatives. This observation allows us to pass to the limit in the final estimates, and conclude the localization property for the limiting function:
\begin{equation}\label{eps_sol}
u(x,t)=\lim_{\epsilon\to 0}u_{\epsilon}(x,t),
\end{equation}   
which is considered as a weak passage to the limit \cite{CIL92}. Thus, we define the weak viscosity solution from \cref{eps_sol} as a partial limit with respect to the norm that will be induced by \cref{I_n} in \cref{sec:localization}. 
This solution may not be unique, and we will assume \textit{a priori}, in \cref{sec:localization}, that the limiting function $u(x,t)$ has certain properties.

Next, in this section, we prove a maximum principle, which evidently will not depend on the parameter $\epsilon$, therefore the limiting function is bounded. In this article, we assume that this limit exists. Then, the obtained localization property will extend from the  estimates on $u_{\epsilon}$ to the weak solution $u$.

\begin{lemma}[Weak maximum principle]\label{maximum-principle_weak} 
For a given function $u_{\epsilon} \in \mathcal{C}^{2,1}_{x,t}$, if $L_{\epsilon}u_{\epsilon} > 0$ in $U_T$, then $u_{\epsilon} \geq \min_{\Gamma(U_T)} u_{\epsilon}  $ in $U_T$.
\end{lemma}

\begin{proof} 
Let us suppose there is a point $(x_0, t_0) \in U_T$ s.t.\ $\min_{\overline{U_T}}u_{\epsilon}=u_{\epsilon}(x_0, t_0)< \min_{\Gamma(U_T)} u$, then the minimum of $u_{\epsilon}(x, t)$ attained at $(x_0, t_0) \in U_T$, for $x_0 \in U$ and $t_0 \leq T$.
Therefore, at the point $(x_0, t_0)$, $\nabla u_{\epsilon}= 0$, $\frac{\partial u_{\epsilon}}{\partial t} \leq 0$ and $\Delta u_{\epsilon} \geq 0$.

Consequently, $L_{\epsilon}u_{\epsilon} \leq 0$ at the point $(x_0, t_0)$, which is a contradiction. 
\end{proof}

\begin{theorem}[Strong maximum principle]\label{maximum_principle_strong}   
For a given function $u_{\epsilon} \in \mathcal{C}^{2,1}_{x,t}$, if $L_{\epsilon}u_{\epsilon} \geq 0$ in $U_T$, then $u_{\epsilon} \geq \min_{\Gamma(U_T)} u_{\epsilon}$ in $U_T$.
\end{theorem}
\begin{proof}
Consider the function
$w(x, t)= Kt+ u_{\epsilon}(x, t)$ for $(x, t) \in U_T$ and  $t<T$, $K >0$. 
Then for given $u_{\epsilon}$, 
\begin{equation}
\begin{split}
    \tilde{L}_{\epsilon}& w \triangleq \tau_0\frac{\partial w}{\partial t}+ (u_{\epsilon}+\epsilon)^{\gamma}(|\nabla u_{\epsilon}|^{\beta}+\epsilon)\vec{\Delta}^{e} \cdot\nabla w-\frac{\sigma^2}{2}\cdot (u_{\epsilon}+\epsilon)^{\gamma}(|\nabla u_{\epsilon}|^{\beta}+\epsilon)\Delta w\\
    &= K\tau_0+\tau_0\frac{\partial u_{\epsilon}}{\partial t}+(u_\epsilon+\epsilon)^{\gamma}(|\nabla u_{\epsilon}|^{\beta}+\epsilon)\vec{\Delta}^{e} \cdot \nabla u_{\epsilon}-\frac{\sigma^2}{2}\cdot (u_{\epsilon}+\epsilon)^{\gamma}(|\nabla u_{\epsilon}|^{\beta}+\epsilon)\Delta u_{\epsilon}\\
    &= K\tau_0+L_{\epsilon}u_{\epsilon} \\
    &>0.
\end{split}
\end{equation}
Since $Lu_{\epsilon} \geq 0$, it follows $L_{\epsilon}w \geq K\tau_0>0$, therefore, we obtain $w(x, t) \geq \min_{\Gamma(U_T)} w$ in $U_T$, for some $K\tau_0>0$, by \cref{maximum-principle_weak}.
Consequently, $K t + u_{\epsilon}(x,t) \geq \min_{\Gamma(U_T)}u_{\epsilon}$ in $U_T$, which implies $K T + u_{\epsilon}(x,t) \geq \min_{\Gamma(U_T)}u_{\epsilon}$ in $U_T$ for $K$. Taking $K\to0$, $u_{\epsilon}\geq \min_{\Gamma(U_T)} u_{\epsilon}$. 
\end{proof}

\section{Proof of the localization property (finite speed of propagation)}
\label{sec:localization}

The standard method \cite{Kamin1988} of proving the finite speed of propagation (\emph{localization property} of solutions) is based on the construction of an appropriate ``Barenblatt-type'' barrier function, then applying the maximum principle (\cref{sec:max_bd}). This method has advantages for equations in non-divergence form. The proof was revisited by Tedeev and Vespri \cite{ves-ted} in the framework of equations in divergence form by applying the powerful machinery of De Giorgi \cite{DiBenedetto93} for estimation of the oscillation of the solution via integral fluxes, obviating the need to construct a barrier function.
 
In this section, motivated by the approach of Tedeev and Vespri \cite{ves-ted}, we prove the localization property of solutions using the De Giorgi--Ladyzhenskaya iteration procedure \cite{DGV}, for the non-divergence form \cref{non_diveregent_eq}. Let $ u_{\epsilon}\geq 0$ be a classical solution of the following IBVP:
\begin{align}
&\frac{\partial u_{\epsilon}}{\partial t} + \frac{\Delta^{e}}{a} (u_{\epsilon}+\epsilon)^{\gamma}(|\nabla u_{\epsilon}|^{\beta}+\epsilon) \sum_{i=1}^{d} \frac{\partial u_{\epsilon}}{\partial x_{i}}-\frac{\sigma^2}{2a} (u_{\epsilon}+\epsilon)^{\gamma}(|\nabla u_{\epsilon}|^{\beta}+\epsilon) \Delta u_{\epsilon} = 0,
\label{MODEL EQ}\\
&u_\epsilon(x,0) = u_0(x)+\epsilon,\label{In-data-1}\\
&u_{\epsilon}(x,t)|_{\partial U \times (0,T)} = \epsilon.\label{BC-1}
\end{align}
Here, $ |\nabla u | = \sqrt{\sum_{i=1}^{d} |\partial u /\partial {x_i} |^{2}}$. Observe that \cref{MODEL EQ} is \cref{non_diveregent_eq} with $a=\tau_0$. In this section, we use $\tau$ as a generic (``dummy'') integration variable, not to be confused with $\tau$ from \cref{sec:paradigm}. 

Assume that $\supp u_0 \subset B_{R_{0}}(0)$ and $ U \subset \mathbb{R}^d$ are bounded, and consider the sequence of $ r_n =
2r\left(1-1/2^{n+1}\right)$ for $ n=0,1,2,\hdots$ with $r > 2R_0$. Let $ \overline{r}_n  = (r_n +r_{n+1})/{2}$,  $U_n = U \setminus B_{r_n} (0)$, and $\overline U_n = U \setminus B_{\overline r_n} (0)$.  We define $\eta_{n}$ to be a sequence of cut-off functions s.t.\
\begin{equation}
    \eta_n(x) = 
    \begin{cases} 
    0, &\text{ for } x\in  B_{r_n}(0),\\
    1, &\text{ for }  x\in \overline U_n,\\
    | \nabla \eta_n| \leq \frac{c2^n}{r}, &\text{ otherwise.} 
    \end{cases}
\end{equation}
Here, $c>0$ is a constant.

\begin{lemma}\label{main-ineq-u}
Assume that $\gamma > 0$, $\beta \geq 0$, and ${\sigma^2} > 0$. Let $\theta \geq 1$ and $\Delta^{e},p \in \mathbb{R}$ be s.t.\
\begin{align}
    \beta +2 \leq p <  \left(\frac{\theta + \gamma}{\beta+1}\right) - \frac{2|\Delta^{e}|}{\sigma^2}. \label{p-bounds}
\end{align}
Then, for any $t>0$, there exist constants $c_0$ and $\nu$ such that
\begin{equation}
 \sup_{0 < \tau < t }\int_{U_{n+1}} u_{\epsilon} ^{\theta +1} \,dx  + C \int_{0}^{t} \int_{U_{n+1}} u_{\epsilon}^{\theta +\gamma-1}  |\nabla u_{\epsilon}|^{\beta+2} \,dx d\tau \leq D_n \int_{0}^{t} \int_{U_{n}} u_{\epsilon}^{\theta +\gamma+\beta+1} \,dx d\tau +c_0\epsilon^{\nu}t ,
\end{equation}
where we have introduced the definitions
\begin{align}\label{C-defin}
   C &\triangleq \frac{(\theta +1)}{2a}\left[ {\sigma^2} \left(\frac{\theta + \gamma}{1+\beta}\right) -{2|\Delta^{e}|} -  {{\sigma^2} p}\right],\\  
   D_{n} &\triangleq \frac{(\theta +1)}{2a}\left[ {{\sigma^2} p}\left(\frac{c 2^{n}}{r}\right)^{\beta +2} + 2 {|\Delta^{e}|} \right]. 
   \label{D-defin}
\end{align} 
\end{lemma}
\begin{proof}
Let $ t \leq T$, multiply both sides of \cref{MODEL EQ} by $\eta_n^p u_{\epsilon}^{\theta}$ and integrate over $U_t \triangleq U\times(0,t)$ to obtain:
\begin{equation}
\frac{1}{\theta+1}\int_{U_t} \eta_n^p u_{\epsilon} ^{\theta +1} \,dx +
\frac{\Delta^{e}}{a}\iint_{U_t} \eta_n^p u_{\epsilon}^{\theta +\gamma} |\nabla u_{\epsilon}|^\beta \sum_{i=1}^{d} \frac{\partial u_{\epsilon}}{\partial x_{i}}  \,dx d\tau
+\frac{\sigma^2}{{2a}(\beta +1)}\iint_{U_t}
\nabla\left(\eta_n^p u_{\epsilon}^{\theta +\gamma}\right)|\nabla u_{\epsilon}|^\beta \nabla u_{\epsilon}  \,dx d\tau =  0.
\label{part-int}
\end{equation}
Compute $\nabla\left(\eta_n^p u_{\epsilon}^{\theta +\gamma}\right) $, then \cref{part-int} yields 
\begin{multline}
\frac{1}{\theta+1}\int_{U_t} \eta_n^p u_{\epsilon}^{\theta +1} \,dx -\frac{|\Delta^{e}|}{a}\iint_{U_t} \eta_n^p u_{\epsilon}^{\theta +\gamma} |\nabla u_{\epsilon}|^{\beta+1} \,dx d\tau  
+ \frac{\sigma^2}{2a}\left(\frac{\theta +\gamma}{\beta +1}\right)\iint_{U_t} \eta_n^{p} |\nabla u_{\epsilon}|^{\beta+2}  u_{\epsilon}^{\theta +\gamma-1}  \, dx d\tau \\
\leq
{\frac{\sigma^2}{2a} p}\iint_{U_t} \eta_n^{p-1}| \nabla \eta_n| |\nabla u_{\epsilon}|^{\beta+1} u_{\epsilon}^{\theta +\gamma} \, dx d\tau .
\label{E6}
\end{multline}
Apply Young's inequity:
\begin{equation}
     \eta_n^{p-1}{}| \nabla \eta_n| |\nabla u_{\epsilon}|^{\beta+1} u_{\epsilon}^{\theta +\gamma} 
     \leq 
     \eta_n^{\frac{(p-1)(\beta+2)}{\beta+1}} |\nabla u_{\epsilon}|^{\beta+2}  u_{\epsilon}^{\theta +\gamma-1} + | \nabla \eta_n|^{\beta+2} u_{\epsilon}^{\theta +\gamma+\beta+1}.
\end{equation}
Note that
\begin{equation}
     |\nabla u_{\epsilon}|^{\beta+1} u_{\epsilon}^{\theta +\gamma}
     \leq 
      |\nabla u_{\epsilon}|^{\beta+2}  u_{\epsilon}^{\theta +\gamma-1} +  u_{\epsilon}^{\theta +\gamma+\beta+1}.
\end{equation}
Then, the estimate~\eqref{E6} becomes
\begin{multline}
\frac{1}{\theta+1} \int_{U_t} \eta_n^p u_{\epsilon} ^{\theta +1} \, dx - \frac{|\Delta^{e}|}{a}\iint_{U_t} \eta_n^p u_{\epsilon}^{\theta +\gamma-1} |\nabla u_{\epsilon}|^{\beta+2}  +  \eta_n^p u_{\epsilon}^{\theta +\gamma + \beta+1} \, dx d\tau  +
\frac{\sigma^2}{2a}\left(\frac{\theta +\gamma}{\beta +1}\right)\iint_{U_t} \eta_n^{p} |\nabla u_{\epsilon}|^{\beta+2}  u_\epsilon^{\theta +\gamma-1}  \, dx d\tau  \\
\leq
{\frac{\sigma^2}{2a} p}
\iint_{U_t} \eta_n^{{(p-1)}({\frac{\beta+2}{\beta + 1}})} u_{\epsilon}^{\theta +\gamma-1}  |\nabla u_{\epsilon}|^{\beta+2}   +   u_{\epsilon}^{\theta +\gamma+\beta+1} |\nabla \eta_n|^{\beta +2} \, dx d\tau .
\label{Rearr}
\end{multline}
Rearranging the above inequality, we obtain 
\begin{equation}
    \left[  \frac{\sigma^2}{2a}\left(\frac{\theta +\gamma}{\beta +1}\right) -\frac{|\Delta^{e}|}{a} \right] \eta_n^{p} > {\frac{\sigma^2}{2a} p} \eta_n^{{(p-1)}\left({\frac{\beta+2}{\beta + 1}}\right)} 
\end{equation} 
by \cref{p-bounds}. Note that $U_{n+1} \subset \overline{U}_{n} \subset U_{n} \subset U$. So, the inequality \eqref{Rearr} becomes 
\begin{multline}
\frac{1}{\theta+1} \int_{U_{n+1}}   u_{\epsilon}^{\theta +1} \, dx +  \int_{0}^{t}\int_{U_{n+1}}
\left[ \frac{\sigma^2}{2a}\left(\frac{\theta +\gamma}{\beta +1}\right) -\frac{|\Delta^{e}|}{a}  -  {\frac{\sigma^2}{2a} p}  \right]   u_{\epsilon}^{\theta +\gamma-1}  |\nabla u_{\epsilon}|^{\beta+2} \, dx d\tau \\
\leq
{\frac{\sigma^2}{2a}}p\iint_{U_{n} \setminus {\overline{U}_{n}}}  u_{\epsilon}^{\theta +\gamma+\beta+1}\left(\frac{c 2^{n}}{r}\right)^{\beta +2} \, dx d\tau
+ 
{\frac{|\Delta^{e}|}{a}}\iint_{{U}_{n}}  u_{\epsilon}^{\theta +\gamma + \beta + 1} \, dx d\tau.
\end{multline}
Thus, we have
\begin{equation}
 \sup_{0 < \tau < t }\int_{U_{n+1}} u_{\epsilon} ^{\theta +1} \,dx  +  C \int_{0}^{t} \int_{U_{n+1}} u_{\epsilon}^{\theta +\gamma-1}  |\nabla u_{\epsilon}|^{\beta+2} \, dx d\tau
\leq 
D_{n} \int_{0}^{t} \int_{U_{n}} u_{\epsilon}^{\theta +\gamma+\beta+1} \, dx d\tau.
\label{EQ22}
\end{equation} 
\end{proof}

Now, let us define the mapping
\begin{equation}
 w \triangleq  u_{\epsilon}^{{(\theta +\gamma + \beta + 1)}/{(\beta +2)}},
 \label{w - u}
\end{equation}
and, for any $n = 0,1,2,\hdots$, define 
\begin{equation}\label{I_n}
I_{n} (t) \triangleq  \sup_{0 < \tau < t } \int_{U_{n+1}} w^{q} \,dx  + \int_{0}^{t} \int_{U_{n+1}} |\nabla w|^{\beta+2} \, dx d\tau,
\end{equation}
where $q \triangleq{(\theta +1)(\beta +2)}/{(\theta +\gamma + \beta + 1)}$. 
\begin{lemma}[Iteration lemma]\label{lad-lemma1}
Let $u_{\epsilon}(x,t)\geq 0$ be a classical solution of IBVP \eqref{MODEL EQ}--\eqref{BC-1}. Then $ \exists \;  C_L > 0$,  $b_L > 1$, $c$, $a>0$ and $\nu>0$ s.t.\ for given $n=1,2,\hdots $,
\begin{equation}\label{Lad-itn-ineq}
I_{n}^{\epsilon}(t) \leq t^{1-\zeta} C_L b^{(n-1)}_L \left(I_{n-1}^{\epsilon}(t)\right)^{1+\epsilon_0}+c\epsilon^\nu t^a,
\end{equation}
where
\begin{align}\label{epsilon-0}
  \epsilon_0 &\triangleq (1-\zeta) \left( \frac{\beta+2}{q} -1\right),\\
    \label{zeta}
    \zeta &\triangleq \frac{\gamma + \beta}{\gamma + \beta +d(\beta +2)(\theta +1)}.   
\end{align}
\end{lemma}
\begin{proof}
The inequality \eqref{EQ22} becomes
\begin{equation}
\sup_{0 < \tau < t }\int_{U_{n+1}} w^{q} \,dx  +  \int_{0}^{t} \int_{U_{n+1}} |\nabla w|^{\beta+2} \,dx d\tau
\leq 
\frac{D_n}{K} \int_{0}^{t} \int_{U_{n}} w^{\beta +2} \,dx d\tau, \label{bridge}
\end{equation}
where  $ K = \min\left\{1,C [q/(\theta +1)]^{\beta +2}\right\}$.

Applying the Gagliardo--Nirenberg--Sobolev inequality on the function $w(x,t)$, one can obtain
\begin{equation}
\int_{U_{n}} w^{\beta+2} \,dx 
 \leq c_G \left(\, \int_{U_{n}} |\nabla w|^{\beta+2} \,dx \right)^{\zeta}\left(\, \int_{U_{n}} w^{q} \,dx\right)^\frac{(1-\zeta)(\beta+2)}{q}
 \label{N-G}
\end{equation}
for some constant $c_G$.
Integrating the above inequality over time, we get
\begin{equation}
\int_{0}^{t}
\int_{U_{n}} w^{\beta+2} \,dx d\tau
\leq c_G \int_{0}^{t} \left(\, \int_{U_{n}} |\nabla w|^{\beta+2} \,dx \right)^{\zeta}d\tau
\left(\sup_{0 < \tau < t } \int_{U_{n}} w^{q} \,dx\right)^\frac{(1-\zeta)(\beta+2)}{q}.
\end{equation}
By applying H\"{o}lder's inequality in the last result, we have
\begin{equation}
\int_{0}^{t}\int_{U_{n}} w^{\beta+2} \,dx d\tau
\leq
c_G t^{1-\zeta} \left( \int_{0}^{t} \int_{U_{n}} |\nabla w|^{\beta+2} \,dxd\tau \right)^{\zeta} 
\left(\sup_{0 < \tau < t } \int_{U_{n}} w^{q} \,dx\right)^\frac{(1-\zeta)(\beta+2)}{q}.
\label{bridge2}
\end{equation}
Using \cref{bridge2} in \cref{bridge}, the quantity defined in \cref{I_n} is bounded as 
\begin{equation}\label{pre-ite}
I_{n}(t) \leq 
\frac{D_{n}}{K}
c_G t^{1-\zeta} I_{n-1}^{\zeta + \frac{(1-\zeta)(\beta+2)}{q}}. 
\end{equation}
Note that ${\zeta + \frac{(1-\zeta)(\beta+2)}{q}} = 1 + \epsilon_0$. Moreover, by \cref{p-bounds}, we obtain
\begin{equation}\label{N-ineq}
D_n \leq (\theta +1) \left[\frac{|\Delta^{e}|}{a} + \left(\frac{c}{2R_{0}} \right) \frac{\sigma^2}{2a} (\theta +\gamma)\right] 2^{{n\beta}+2n} .
\end{equation}
Therefore, from \cref{pre-ite}, we have the following iterative relation
\begin{equation}\label{ite-sheme}
    I_{n} \leq t^{1-\zeta}C_Lb^{n-1}_L I_{n-1}^{1+\epsilon_0}
\end{equation}
with $C_L=\left\{ (\theta +1) \left[\frac{|\Delta^{e}|}{a} + \left(\frac{c}{2R_{0}} \right) \frac{\sigma^2}{2a} (\theta +\gamma)\right]\frac{c_G}{K} 2^{{\beta}+2}\right\}$ and $b_L =2^{\beta+2}$.
\end{proof}

Finally, the main inequality of \cref{lad-lemma1} follows from \cref{ite-sheme}. 
It is not difficult to show that $\exists~\mu>0$ depending on the parameters of the model s.t.\ for the integral
\begin{equation}
    \tilde{I}_n^{\epsilon}=\sup_{0 < \tau < t } \int_{U_{n+1}} u_{\epsilon}^{\mu} \,dx,
\end{equation}
the corollaries follow.
\begin{corollary}
There exist constants $\nu>0$, $\mu>0,$ $a$, and $c$ independent from $\epsilon$ s.t.\ for any given $n$,
\begin{equation}
\tilde{I}_n^{\epsilon}\leq b^n t^{\epsilon_0} \left(\tilde{I}_{n-1}^{\epsilon}\right)^{1+\epsilon_0} + c t^{a} \epsilon^{\nu}. 
\end{equation}
\end{corollary}
\begin{corollary}
Let $u$ be a weak viscosity solution in the sense of weak convergence of the sequence regularized solutions $u_{\epsilon}$ and
\begin{equation}
\tilde{I}_n=\sup_{0 < \tau < t }\int_{U_{n+1}} u^{\mu} \,dx,
\end{equation}
then
\begin{equation}
\tilde{I}_n\leq b^nt^{\epsilon_0}\left(\tilde{I}_{n-1}\right)^{1+\epsilon_0}.
\end{equation}
\end{corollary}

Finally, via the Ladyzhenskaya iterative Lemma \cite{LAU}, the following theorem is established.
\begin{theorem}\label{In_theorem}
Assume the weak viscosity solution  $u(x,t)$ is  s.t.\
\begin{equation}\label{I_oassump}
\tilde{I}_{0}(T) \leq  C_L^{-\frac{1}{\epsilon_0}} 2^{-\left(\frac{\beta+2}{\epsilon_0^2}\right)} T^{-\left(\frac{\theta+1}{\gamma+\beta}\right)} .
\end{equation}
Then,
\begin{equation}\label{In-to-0}
\tilde{I}_{n}(T)\rightarrow 0 \text{ as } n\rightarrow \infty.
\end{equation}
\end{theorem}

\begin{corollary}[Localization property of solutions]
Since $U _{n+1}\subset U_{n}$, it follows from \cref{In_theorem} and the maximum principle that, for any initial data with compact support in the ball $B_{2r}(0)$, there exists $T$ s.t.\ the weak viscosity solution
\begin{equation} 
u (x,t)= 0 \quad \text{ a.e. } \quad \text{ in } \quad  U\setminus{B_{2r}(0)} \quad  \text{ for any } \quad   t\leq T.
\end{equation}
\end{corollary} 

\section{Mapping of the non-divergence form equation to divergence form}
\label{sec:mapping}

The physical models in which nonlinear degenerate diffusion equations arise, which were discussed in \cref{sec:intro}, generically lead to \emph{divergence form} equations. On the other hand, however, the Einstein paradigm, introduced in \cref{sec:paradigm}, generically leads to a PDE in non-divergence form. To connect these two formulations, in this section, we present a transition formula between them, which is often convenient in the analysis of degenerate parabolic PDE \cite{Aronson1983}.

As it was pointed out in \cref{sec:intro}, the solution of the degenerate parabolic PDE arising from the Einstein paradigm for Brownian motion may not be smooth. So, our estimates and localization property did not rely on smoothness. To be clear, however, we  worked with the so-called viscosity solution \cite{CIL92}, and obtained a  qualitative result. In this section, we are assuming that the solution of both the divergent and non-divergent form equation is smooth enough to prove a mapping formula between them. We believe that it will not difficult to show that the weak viscosity solution will obey such a mapping as well. However, this is not a result needed for the upcoming illustrations of the finite speed of propagation (localization) property, so we leave this detail for future work.

Since all our semi-analytic solutions and numerical simulations, which follow in \cref{sec:numerical}, will be performed on a 1D spatial domain without drift, we henceforth assume that $x\in\mathbb{R}$ and $\Delta^e=0$. The key mathematical result of this section is to connect solutions of a generalized degenerate Einstein equation in non-divergence form \cref{non-div-eq} to solutions of a corresponding \cref{div-eq} in divergence form.

\begin{theorem}[Mapping of solutions]\label{mapping-theorem}
Let $0\leq \gamma<1.$ Assume that  $u(x,t)\in \mathcal C^{2,1}_{x,t}$, $u\geq 0$, and it solves
\begin{equation}\label{non-div-eq}
Lu=\frac{\partial u}{\partial t}-\frac{\sigma^2}{2} u^{\gamma}\left|\frac{\partial u}{\partial x} \right|^{\beta}\frac{\partial^2 u}{\partial x^2}=0.
\end{equation}
Let 
\begin{equation}\label{alph-gamma}
\alpha=\frac{\gamma}{1-\gamma} \quad\Leftrightarrow\quad \gamma=\frac{\alpha}{1+\alpha}.       
\end{equation}
Define a $u$-$v$ mapping via
\begin{equation}\label{alpha-def}
v \triangleq u^{\frac{1}{\alpha+1}}.
\end{equation}
Then, $v(x,t)\ge 0$ is the solution of
\begin{equation}\label{div-eq}
\hat{L}v=\frac{\partial v}{\partial t} - q_0\frac{\partial}{\partial x}\left(v^{\gamma_0}\left|\frac{\partial v}{\partial x}\right|^{\beta}\frac{\partial v}{\partial x}\right)=0,
\end{equation}
where
\begin{equation}\label{gama-alpha-def}
q_0=\frac{\sigma^2}{2}(\alpha+1)^{\beta}   \qquad \text{and}  \qquad \gamma_0=\alpha(\beta+1).
\end{equation}
\end{theorem}
\begin{proof} We proceed as follows.
\begin{enumerate}
    \item Since we have assumed that $x\in \mathbb{R}$, then $|\nabla v|= \sqrt{(\partial v/\partial x)^2}$. Additionally, for any function $v\in \mathcal{C}^1_x$ with $\gamma_0$ given in \cref{gama-alpha-def}, it is straightforward to show that
\begin{equation}\label{ux-u-gamm}
\frac{1}{(\alpha+1)^{1+\beta}}\left|\frac{\partial (v^{\alpha +1})}{\partial x}\right|^\beta \frac{\partial (v^{\alpha +1})}{\partial x} 
= v^{\gamma_0}\left|\frac{\partial v}{\partial x}\right|^\beta \frac{\partial v}{\partial x},
\end{equation}
where we have used the fact that $v\ge0$.
\item
From the above, it follows that, for some constant $C_1$,
\begin{equation}\label{Lu-mod}
\begin{split}
L_0v &=\frac{\partial v}{\partial t}-C_1 \frac{\partial }{\partial x}\left(v^{\gamma_0}\left|\frac{\partial v}{\partial x}\right|^\beta \frac{\partial v}{\partial x}\right) \\
&= \frac{\partial v}{\partial t}-C_1\frac{1}{(\alpha+1)^{1+\beta}} \frac{\partial}{\partial x}\left(\left|\frac{\partial (v^{\alpha +1})}{\partial x}\right|^\beta \frac{\partial (v^{\alpha +1})}{\partial x}\right) .
\end{split}
\end{equation}
\item
Multiplying \cref{Lu-mod} by $v^{\alpha}$ and using the mapping \eqref{alpha-def}, we obtain
\begin{equation}\label{Lu-mod-1}
\begin{split}
v^\alpha L_0v &= v^{\alpha}\frac{\partial v}{\partial t}-C_1\frac{1}{(\alpha+1)^{1+\beta}} v^{\alpha} \frac{\partial}{\partial x}\left(\left|\frac{\partial (v^{\alpha +1})}{\partial x}\right|^\beta \frac{\partial (v^{\alpha +1})}{\partial x}\right) \\
&= \frac{1}{(\alpha+1)} \frac{\partial u}{\partial t} - C_1 \frac{1}{(\alpha+1)^{1+\beta}} u^{\frac{\alpha}{1+\alpha}} \frac{\partial}{\partial x}\left(\left|\frac{\partial u}{\partial x}\right|^\beta \frac{\partial u}{\partial x}\right).
\end{split}
\end{equation}
Introducing $\gamma$ from \cref{alph-gamma} into \cref{Lu-mod-1}:
\begin{equation}\label{vLv}
    (\alpha+1)v^\alpha L_0 v=\frac{\partial u}{\partial t} - C_1 \frac{1}{(\alpha+1)^{\beta}} u^{\gamma} \frac{\partial}{\partial x}\left(\left|\frac{\partial u}{\partial x}\right|^\beta \frac{\partial u}{\partial x}\right).
\end{equation}
\item 
Observe that 
\begin{equation} \label{uxuxx}
\frac{\partial}{\partial x}\left(\left|\frac{\partial u}{\partial x}\right|^{\beta}\frac{\partial u }{\partial x}\right)
 = (1+\beta)\left|\frac{\partial u}{\partial x}\right|^\beta\frac{\partial^2 u}{\partial x^2}.
\end{equation}
\item Substituting \cref{uxuxx} into \cref{vLv}, we obtain
\begin{equation}\label{vLvu}
    (\alpha+1)v^\alpha L_0 v=\frac{\partial u}{\partial t} - C_1 \frac{(1+\beta)}{(\alpha+1)^{\beta}} u^{\gamma}\left|\frac{\partial u}{\partial x}\right|^\beta \frac{\partial^2 u}{\partial x^2}.
\end{equation}
\item Suppose that $L_0v=0$ and let $C_1=\frac{q_0}{1+\beta}$, then we obtain the desired result.
\end{enumerate}
\end{proof}

\begin{remark}\label{remove-q} 
We can remove the coefficient $q_0$ from \cref{div-eq} by letting $(x',t')=(x,q_0t)$, where $q_0$ is given in \cref{gama-alpha-def}.
\end{remark}


\section{Numerical examples of the localization property of solutions}
\label{sec:numerical}

\subsection{Preliminaries}

In light of \cref{remove-q}, let us relabel $x'\leftrightarrow x$ and $t'\leftrightarrow t$ for simplicity. Now, based on \cref{div-eq} and the above deliberations, consider the 1D problem:
\begin{equation}\label{div-eq-1D}
\frac{\partial v}{\partial t}-\frac{\partial}{\partial x}\left(v^{\gamma_0}\left|\frac{\partial v}{\partial x}\right|^{m}\frac{\partial v}{\partial x}\right)=0,
\end{equation}
where the sign of $m$ is not fixed. The results from \cref{sec:localization} apply for the case of $m=\beta>0$. In this section, we will discuss some cases (beyond the scope of our earlier results) that do not assume $m>0$, hence the change in notation.

\Cref{div-eq-1D} is subject to the \emph{compact} initial condition (IC):
\begin{equation}
v(x,0) = v_0(x) = \begin{cases} 
\displaystyle \frac{3}{4x_0}\left[1 - \left(\frac{x}{x_0}\right)^2\right], &\quad|x|\le x_0,\\ 
0, &\quad|x|>x_0.
\end{cases}
\label{eq:v_ic_box}
\end{equation}
Take $x_0=1$ without loss of generality. Note that we do not take a ``box'' function because then $|dv_0/dx|$ would be undefined as $|x|\to x_0$, which would cause numerical difficulties. Instead, we take a ``mound'' function (inverted parabola) of unit area.

We solve \cref{div-eq-1D} subject to \eqref{eq:v_ic_box} on a finite length domain, $x\in[-x_\mathrm{max},+x_\mathrm{max}]$ subject to ``natural'' (Neumann) boundary conditions (BCs) at the endpoints:
\begin{equation}
    \left.\frac{\partial v}{\partial x}\right|_{x=\pm x_\mathrm{max}} = 0 \qquad \forall t\in [0,T].
\label{eq:Neumann_bc_u}
\end{equation}
The interval $[-x_\mathrm{max},+x_\mathrm{max}]$ is chosen to be large enough, so that the imposed BCs at $x=\pm x_\mathrm{max}$ have no influence whatsoever on the nonlinear diffusion process of a localized initial condition, during the time interval $(0,T]$ of interest. 

\subsection{Explicit construction of the Kompaneets--Zel'dovich--Barenblatt-type barrier function}
\label{sec:self-similar}

Given the localization property proved in \cref{sec:localization}, finite speed of expansion of the compact support of the initial condition is expected under generic conditions without the need to explicitly construct a barrier function. However, for a class of 1D PDEs arising from our nonlinear Einstein paradigm, it is also possible to explicitly construct the spreading self-similar solutions with compact support.

Therefore, for completeness and to illustrate our general mathematical result, in this subsection, we construct solutions with finite speed of propagation explicitly using the Kompaneets--Zel'dovich--Barenblatt self-similarity approach \cite{Zel67,Barenblatt1972,Barenblatt96}. This calculation, of course, is only explicit for the 1D model considered in this section, \textit{i.e.}, \cref{div-eq-1D} subject to $v(x,0) = \delta(x)$ (point source, or fundamental, solution).

To this end, let the support of the solution increase as $[-x_f(t),+x_f(t)]$ for some \emph{moving front} function $x_f(t)$ to be determined. The point-source initial condition together with the localization property of the solution imposes a ``finite mass'' constraint on $v$: 
\begin{equation}\label{eq:finite-mass}
    \int_{-x_f(t)}^{+x_f(t)} v(x,t)\,dx = 1 \quad\forall t\ge0. 
\end{equation}
Next, assume that $v(x,t) = Vt^{-\varsigma} f(\xi)$, where $V$ is a constant, $\xi \triangleq \eta/\eta_f$ is the normalized similarity variable, with $\eta \triangleq xt^{-\varsigma}$ and $\eta_f$ being the (constant) value of $\eta$ at $x=x_f(t)$ (\textit{i.e.}, $\eta_f = x_f(t)t^{-\varsigma}$). Here, $V$, $\varsigma$ and $\eta_f$ must be determined from the PDE~\eqref{div-eq-1D} and any initial/boundary condition(s) to ensure the self-similarity assumption that $v(x,t) = Vt^{-\varsigma} f(\xi)$ holds. Substituting this transformation into \cref{div-eq-1D}, it is easy to show that such a solution exists if the following ordinary differential equation (ODE) is satisfied:
\begin{equation}\label{eq:f-ode}
\varsigma \left(f + \xi f'\right) + (f^{\gamma_0} |f'|^m f')' = 0, \qquad \xi \in(-1,+1),
\end{equation}
with
\begin{equation}\label{eq:nu-A}
     \varsigma = \frac{1}{\gamma_0 + 2m + 2}, \qquad V = \eta_f^{(m+2)/(\gamma_0+m)}.
\end{equation}
Here, prime denotes differentiation with respect to the argument of $f$, \textit{i.e.}, $d/d\xi$. By the localization property of the solution, $f(\pm1)=0$. Finally, $\eta_f$ is determined by transforming \cref{eq:finite-mass}:
\begin{equation}\label{eq:etaf}
    \eta_f = \left[\int_{-1}^{+1} f(\xi) \, d\xi\right]^{-(\gamma_0+m)/(\gamma_0+2m+2)},
\end{equation}
having used the expression for $V$ from \cref{eq:nu-A}.

Next, observe that the left-hand side of the ODE \eqref{eq:f-ode} is a product rule, hence a first integral exists:
\begin{equation}\label{eq:f-first-ingetral}
    \varsigma \xi f+ f^{\gamma_0} |f'|^m f' = 0,
\end{equation}
where the constant of integration is set to zero by requiring $\xi \leftrightarrow -\xi$ symmetry. Symmetry also allows us to deal with the absolute value in \cref{eq:f-first-ingetral} by restricting to $\xi\ge 0$. The simulations to be discussed in \cref{sec:numerical_examples} suggest that $f(\xi)$ is monotonically decreasing on $[0,1]$, so we can set $|f'| = -f'$. Then, the first integral \eqref{eq:f-first-ingetral} of the ODE becomes
\begin{equation}
    f^{(\gamma_0-1)/(m+1)} f' = -\left(\varsigma \xi\right)^{1/(m+1)}.
\end{equation}
Now, using $f(1)=0$ and $\varsigma$ from \cref{eq:nu-A}, the solution to the ODE~\eqref{eq:f-ode} is
\begin{equation}\label{eq:f-soln}
  f(\xi) = \left[ \left(\frac{\gamma_0 +m}{m+2}\right)  (\gamma_0 + 2m + 2)^{\frac{-1}{m+1}} \left(1 - \xi^{\frac{m+2}{m+1}} \right) \right]^{\frac{m+1}{\gamma_0 +m}}, \quad \xi\in[0,1].
\end{equation}
From the constraint~\eqref{eq:etaf} it follows that
\begin{equation}
    \eta_f = \left(\frac{m+2}{\gamma_0+m}\right)^{\frac{m+1}{\gamma_0+2m+2}} (\gamma_0+2 m+2)^{\frac{1}{\gamma_0+2m+2}} 
    \left[\frac{1}{2}\left(\frac{m+2}{m+1}\right)\frac{\Gamma \left(\frac{m+1}{m+2}+\frac{m+1}{m+\gamma_0}+1\right)}{\Gamma\left(\frac{m+1}{m+2}\right) \Gamma \left(\frac{m+1}{m+\gamma_0}+1\right)} \right]^{\frac{\gamma_0+m}{\gamma_0+2m+2}},
\end{equation}
where $\Gamma(z)$ is the Gamma function. See also the works by Eggers and Fontelos \cite{Eggers2015} (Exercise 3.5, p.~59 therein) or Pattle  \cite{Pattle1959} for the (classical) special case of $m=0$.
 
Motivated by Barenblatt \cite{Barenblatt52,Barenblatt52b}, we observe that the behavior of the gradient $f'$, as $\xi\to0$ or $\xi\to1$, depends on $\gamma_0$ and $m$. From the self-similar solution~\eqref{eq:f-soln},  we compute
\begin{equation}\label{eq:f-prime}
    f'(\xi) = -(\gamma_0+2 m+2)^{\frac{-1}{m+1}} \xi ^{\frac{1}{m+1}} \left[ \frac{(\gamma_0+m)}{(m+2)(\gamma_0+2 m+2)^{\frac{1}{m+1}}} \left(1-\xi ^{\frac{m+2}{m+1}}\right) \right]^{\frac{1-\gamma_0}{\gamma_0+m}}
\end{equation}
to understand the possible scenarios.
In \cref{sec:localization}, it was assumed that $m>0,\gamma_0\ge0$. Here, let us consider all possibilities instead. From \cref{eq:f-prime}, it follows that the behavior of $f'(\xi)$ is regular as $\xi\to0$ and $f'(0)=0$ $\forall \gamma_0$ if $m>-1$ and singular with $|f'(\xi)| \to \infty$ as $\xi\to0^+$ if $m<-1$. The special case of $m=-1$, if it is of interest, appears to be problematic and requires more careful analysis. Meanwhile, as $\xi\to1^-$, $f'(\xi)$ is regular and $f'(1) = 0$ if $(1-\gamma_0)/(\gamma_0 + m) >0$ and singular with $f'(\xi) \to - \infty$ as $\xi\to1^-$ if $(1-\gamma_0)/(\gamma_0 + m) <0$. The special case of $\gamma_0=1$ is interesting because then $f'(1) = -(2 m+2)^{\frac{-1}{m+1}} <0$, assuming $m>-1$.

These observations also highlight that the self-similar solution is only a weak solution of the original PDE~\eqref{div-eq-1D} (see, \textit{e.g.}, Ref.~\cite{Gilding76} and related literature).

Next, we numerically verify the localization property of solutions (\textit{i.e.}, the main mathematical result on finite speed of propagation), as well as the predictions from this section based on the self-similarity analysis.

\subsection{Numerical results}
\label{sec:numerical_examples}

The self-similar solution $v(x,t) = Vt^{-\varsigma}f(\xi)$ based on \cref{eq:f-soln}, satisfying the normalization condition~\eqref{eq:finite-mass}, is strictly speaking the fundamental solution \cite{Kamin1989}, and it is not the exact solution for the initial data~\eqref{eq:v_ic_box}. However, from the principle of \emph{intermediate asymptotics} \cite{Barenblatt1972,Barenblatt96}, we know that a solution of \cref{div-eq-1D} subject to \eqref{eq:v_ic_box} eventually converges (for large $t\to\infty$) to the point-source self-similar solution, as the initial data is ``forgotten.'' That is, the self-similar solution from \cref{sec:self-similar} is, in a sense, an \emph{attractor}. Hence, the support of the solution $v(x,t)$ starting from the compactly supported initial condition~\eqref{eq:v_ic_box} remains compact for $t>0$. In other words, $v\big(|x|\ge x_f(t),t\big) = 0$ while $v\big(|x|<x_f(t),t\big) \ne 0$ (given the $x$-symmetry of the problem).  Importantly, from the principle of intermediate asymptotics, we know that the compact support of this non-fundamental solution will (for large $t$) expand with the \emph{same} finite speed of propagation as the fundamental (self-similar) solution.

Next, to verify these claims, we use the strongly-implicit finite-difference scheme in conservative form developed and benchmarked by Ghodgaonkar and Christov \cite{gc19}. This scheme is well suited for the types of gradient-degenerate diffusion equations~\eqref{div-eq-1D} considered in this work. We ensured that the reported numerical results are grid and time-step independent. After obtaining the numerical solution of the PDE~\eqref{div-eq-1D} subject to \eqref{eq:v_ic_box}, we found $x_f(t)$ s.t.\ $v\big(x = \pm x_f(t),t\big) = 0$. This quantity can then be compared to the prediction of the self-similarity analysis. Specifically, $x_f(t) \sim t^{1/(\gamma_0+2m+2)}$, or $\log x_f(t) \sim \frac{1}{\gamma_0+2m+2}\log t$, based on \cref{sec:self-similar}.

\begin{figure}[ht!]
\centering
\subfigure[]{\includegraphics[width=0.75\textwidth]{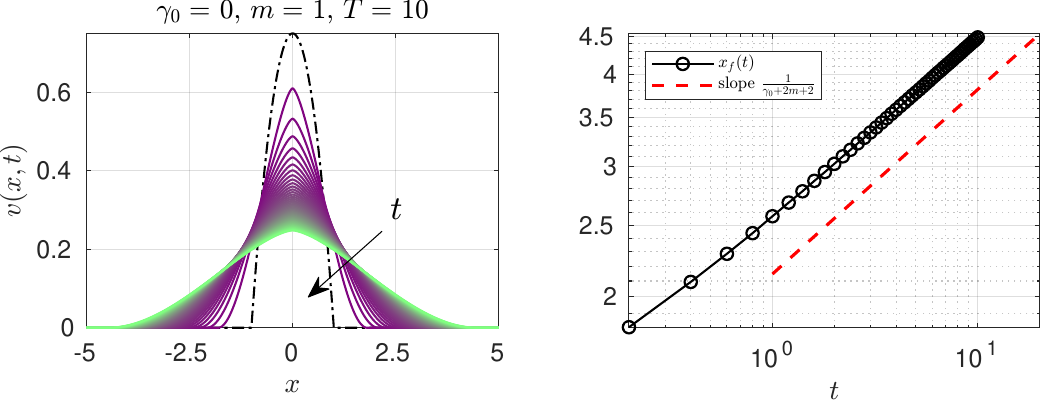}}
\subfigure[]{\includegraphics[width=0.75\textwidth]{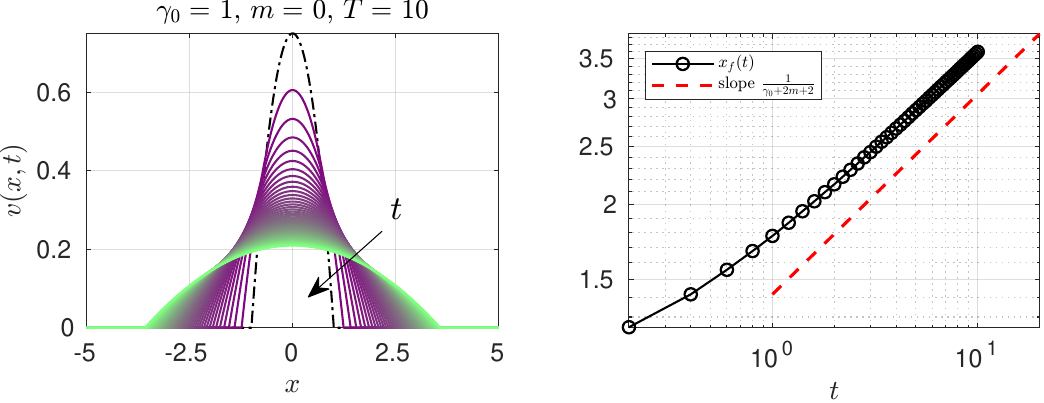}}
\subfigure[]{\includegraphics[width=0.75\textwidth]{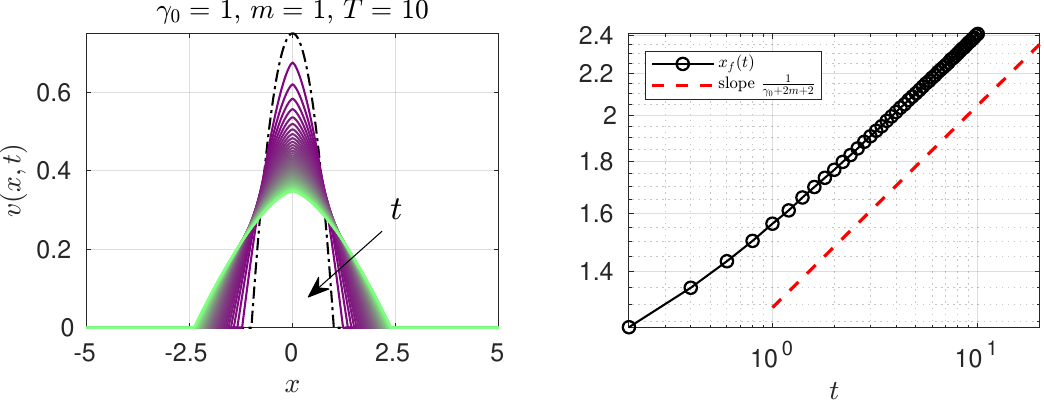}}
\caption{Example numerical simulations showing finite speed of propagation of initially compact solutions of the PDE~\eqref{div-eq-1D} from the initial condition~\eqref{eq:v_ic_box} (shown as dash-dotted curves in left panels), for different choices of nonlinearity exponents $\gamma_0$ and $m$, up to $t=T=10$. The speed of propagation, as predicted by the analysis in \cref{sec:self-similar}, is represented by the slope of the dashed trendlines in the right panels. Note that (b) is the special case corresponding to the ``Barenblatt solution'' of the porous medium equation. The time evolution in the left panels is in the direction of the arrow from dark to light colored curves. The plots in the right panels are on a log-log scale.}
\label{fig:simulation}
\end{figure}

The numerical results are shown in \cref{fig:simulation}. We confirm the finite speed of propagation of the solution at the speed predicted by self-similarity analysis.

Also, observe the sharp crest of the solution at $x=0$ for $m=1$ in \cref{fig:simulation}(a,c). Indeed, \cref{eq:f-prime} suggests that $\partial v/\partial x \sim \sgn(x) |x|^{1/2}$ as $|x|\to0$, hence there's a curvature singularity at $x=0$ leading to the sharp crest. Additionally, as discussed in \cref{sec:self-similar}, $\partial v/\partial x \to 0$ as $|x|\to x_f(t)$ if $\gamma_0\ne0$ (\cref{fig:simulation}(a)), but $\partial v/\partial x \to \sgn(x) const.$ as $|x|\to x_f(t)$ in the special case of $\gamma_0=1$ (\cref{fig:simulation}(b,c)).

\section{Discussion on the localization property}
\label{sec:discussion}

The early references on the localization property for parabolic equations link to the work of Zel'dovich. The localization property was then proved by Barenblatt by constructing the ``Barenblatt type'' of barrier function for the corresponding degenerate nonlinear PDE. They understood that, if the diffusivity degenerates with respect to the solution, then the solution exhibits finite speed of propagation, which can provide a more realistic interpretation of physical observations typically understood via a Darcy/Fourier/Fick-type law. Specifically, with degenerate PDEs, the localization property of the solution of the evolutionary equation is called, in some publications, a ``dead zone." In other words, a non-negative solution of the steady-state evolutionary equation, if it exists, must be zero outside of some compact support. Independently, this property of solutions of elliptic equations with nonlinear absorption (dependent on the solution only) was investigated by Landis \cite{Landis,Landis1993} using his method of lemma of the growth of a narrow domain. This approach leads to a very different perspective on the localization property of solutions.

In the recent work by Tedeev and Vespri \cite{ves-ted}, a method based on De Giorgi's iteration (see also the Ladyzhenskaya iterative process) was employed to prove this essential feature for a class of degenerate parabolic equation in divergence form. We used the Tedeev and Vespri approach as our groundwork, but we provided a proof of the localization property (\cref{sec:localization}) for degenerate parabolic equation in \emph{non-divergence form}, which we derived from a generalized Einstein Brownian motion paradigm (\cref{sec:paradigm}), giving the non-divergence form equations a new physical interpretation. It is appropriate to mention that Landis used his method to provide an alternative proof of De Giorgi's celebrated theorem on H\"{o}lder continuity of the solution of elliptic equation of second order \cite{Landis67}. We believe that, by employing the Landis method, in future work we can significantly widen the class of equations for which the localization property holds. In this way, we can also extend the present results by keeping the absorption term in \cref{dif-drift-absorp-generic}. Importantly, when starting from the non-divergence form of the equation under the Einstein paradigm, we do not require the smoothness of coefficients, which was assumed in previous works. 

\begin{figure}[ht!]
\centering
\includegraphics[scale=0.7]{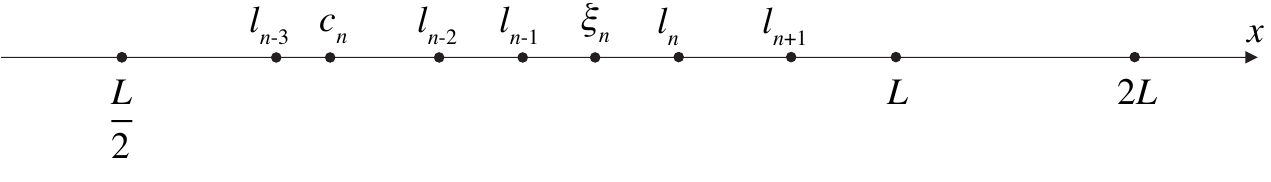}
\caption{Schematic illustrating of the location of the points on the $x$-axis.}
\label{fig:schematic-points}
\end{figure}

Finally, we sketch the arguments that underlay Landis' machinery for proving the localization property of solutions to the 1D degenerate Einstein model. As it was explained above, for the time interval of free jumps in the most basic nonlinear model Einstein, one can assume that $\tau = \tau_0 u^{-\gamma}$ ($\tau_0$ can be absorbed into $t$ and scaled out). Here, $0\leq \gamma<1$ and $u\geq 0$. Then, one can rewrite the model as follows:
\begin{align}
 &L_0u=
 \frac{\partial u^{1-\gamma}}{\partial t}-\frac{\partial^2 u}{\partial x^2}=0 \; \text{ in the domain } \; 0<x<2L, \; \text{ for } \; 0<t<\infty, \label{L0-eq}\\
 &u(x,0)=u_0(x)>0 \ \text{ for } \ {0<x<\frac{L}{2}},\label{non-zero  data} \\
& u(x,0)=0 \; \text{ for } \; {\frac{L}{2}\leq x < 2 L},\label{zero data}\\
 &\frac{\partial u}{\partial x}(0,t)=\frac{\partial u}{\partial x}(2L,t)=0 \; \text{ for } \; 0<t<\infty .\label{Neumann}
\end{align}

It will be convenient to use the notation introduced in the diagram in  \cref{fig:schematic-points}.
Let $l_n=L-2^{-n}$ and $I_n=\{x:2L>x>l_n\}$ and $J_n=I_{n}\setminus I_{n+1}$. In the layers $J_{n-1}$ and $J_{n-2}$ introduce auxiliary points $\xi_n\in J_{n-1}$ and $c_n\in J_{n-2}$, which will be specified later. First observe that, due to initial and boundary conditions, the function $v(x)=\int_0^T u(x,\tau)\,d\tau$ satisfies a maximum principle: $v(c_n)\geq v(\xi_n)$ for any points $c_n<\xi_n$. After integration by parts in the domain $\tilde{I}_n = I_n \cap\{x>\xi_n\}$ and using the Lagrange mean value theorem with $\xi_n$ as a mean point, one can easily get the following inequality from maximum principle for the function $v(x)$:
\begin{equation}\label{iter-beg}
 \max_{0<\tau<T}\int_{I_{n+1}} u^{1-\gamma}(x,\tau) \,dx \leq \int_{I_n} u^{1-\gamma}(x,T) \,dx \leq 2^n \int_0^T u(\xi_n,\tau) \,d\tau.
\end{equation}
One more application of the maximum principle for $v(x)$ on the domain $\tilde{I}_n=I_{n-2}\cap \{x>c_n\}$ and selecting $c_n=c_n^*$ as an intermediate value point from the integral mean value theorem for the  function $u^{1-\gamma}$, and using arguments as in \cref{1-d-mean-value} with $b=l_{n-3}$ and $a=l_{n-2}$, we obtain      
\begin{equation}\label{iterative-1}
 \int_{I_n} u^{1-\gamma}(x,T) \,dx \leq 2^n \cdot 2^n\int_0^T\left(\,\, \int _{J_{n-2}} u^{1-\gamma}(x,\tau)\,dx\right)^{\frac{1}{1-\gamma}} \,d\tau.
\end{equation}
Then, due to \cref{iter-beg}:
\begin{equation}\label{iterative-2}
 \max_{0<\tau<T}\int_{I_{n+1}} u^{1-\gamma}(x,\tau) \,dx \leq  4^n\cdot T \max_{0<\tau<T}\left(\,\, \int _{I_{n-2}} u^{1-\gamma}(x,\tau)\,dx\right)^{1+\epsilon_0} , \qquad \epsilon_0 = \frac{\gamma}{1-\gamma}.
\end{equation}

From the above inequality follows an iterative estimate similar to the one that we obtained in the iterative inequality in \cref{ite-sheme} above. Consequently, we can conclude the localization property holds for the non-negative  classical solution of the IBVP~\eqref{L0-eq}--\eqref{Neumann} with compact support of the initial data. 

Namely, if initial data in the IBVP \eqref{L0-eq}--\eqref{Neumann} with $0\leq\gamma<1$ are s.t.
\begin{equation}\label{T-model-cond}
    \max_{0<\tau<T}\int_{0}^{2 L} u^{1-\gamma}(x,\tau) \,dx\leq T^{-1/{\epsilon_0}}4^{-1/{\epsilon_0^2}} ,
\end{equation}
then 
\begin{equation}
 u(x,t)\equiv 0 , \text{ for all } x\in[L,2L] \text{ during the time interval } t\in[0,T].     
\end{equation}
\begin{lemma}\label{1-d-mean-value} 
Let $u(x,t)\geq 0$ be a continuous function on $[a,b]\times [0,T]$, and $\delta>0$. For any $T>0$ there exists $a<c^*<b$, s.t.
\begin{equation}\label{int-mean-ineq}
\int_0^Tu^{\delta}(c^*,\tau) \,d\tau \leq \frac{T}{b-a}\max_{0\leq t\leq T}\int_a^b u^\delta (x,t) \,dx.
\end{equation}
\end{lemma}
\begin{proof}
Proof follows from the argument. For any $T>\tau>0$, there exists $a<c(\tau)<b$ s.t.
\begin{equation}\label{u-delta-int-ineq}
 u^\delta \big(c(\tau),\tau\big)=\frac{1}{b-a}\int_a^bu^\delta(x,\tau) \,dx \leq \max_{0\leq t\leq T }\frac{1}{b-a}\int_a^bu^\delta(x,t) \,dx.
\end{equation}
\end{proof}


\section*{Acknowledgments}
We thank Prof.\ J.~R.~Walton for insightful questions about the behavior of the slope of the self-similar solution at the moving front.

\subsection*{Author contributions}
\textbf{Ivan C.\ Christov} --- Conceptualization, Investigation, Project administration, Software, Writing -- original draft. 
\textbf{Isanka Garli Hevage} --- Formal analysis, Investigation, Methodology, Writing -- original draft. 
\textbf{Akif Ibraguimov} --- Conceptualization, Formal analysis, Methodology, Project administration, Supervision, Writing -- original draft. 
\textbf{Rahnuma Islam} --- Formal analysis, Investigation, Methodology, Writing -- original draft. 

\subsection*{Financial disclosure}
None reported.

\subsection*{Conflict of interest}
The authors declare no potential conflicts of interest.

\setlength{\bibsep}{0.0pt}
\bibliography{references.bib}

\end{document}